\documentclass[fleqn,11pt]{article}
\usepackage{graphicx,pict2e,amssymb,amsmath,bm}
\usepackage{ntheorem,delarray}
\usepackage{hyperref,srcltx}
\usepackage[margin=20pt,font=small,labelfont=sc]{caption}
\usepackage{makeidx}
\usepackage[margin=.9in]{geometry}
\usepackage[framemethod=TikZ]{mdframed}

\makeindex\nonstopmode

\long\def\killtext#1{}

\theoremseparator{.}
\newtheorem{theorem}{Theorem}[section]
\newtheorem{prop}[theorem]{Proposition}
\newtheorem{lemma}[theorem]{Lemma}

\newtheorem{claim}[theorem]{Claim}

\theorembodyfont{\rmfamily}

\newtheorem{exa}[theorem]{Example}

\theoremindent.8cm \theorembodyfont{\small}

\newenvironment{proof}{\noindent{\bf Proof.}}{\hfill$\square$\medskip}

\newenvironment{proof*}[1]{\noindent{\bf Proof of #1.}}{\hfill$\square$\medskip}

\def\E{{\sf E}}\def\Var{{\sf Var}}

\def\R{\Rbb}

\def\E{\mathrm{E}}
\def\vol{\mathrm{vol}}
\def\TV{\text{\tiny{TV}}}

\def\Rbb{\mathbb{R}}

\def\vol{\mathrm{vol}}

\def\unif{\mathrm{unif}}

\begin{document}

\title{\Large\bf Stochastic billiards for sampling from the boundary of a convex set}

\author{A. B. Dieker, Santosh S. Vempala}
\maketitle

\begin{abstract}
Stochastic billiards can be used for approximate sampling from the boundary of a bounded convex set through the Markov Chain Monte Carlo (MCMC) paradigm.
This paper studies how many steps of the underlying Markov chain are required to get samples (approximately) from
the uniform distribution on the boundary of the set, for sets with an upper bound on the curvature of the boundary.
Our main theorem implies a polynomial-time algorithm for sampling from the boundary of such sets. 
\end{abstract}

\section{Introduction}
Sampling from high-dimensional objects is a fundamental algorithmic task with many applications to central problems in operations research and computer science. As with optimization, sampling is algorithmically tractable for convex bodies \cite{DyerFK89,lovaszsimonovits:1993,LV3} and their extension to logconcave densities \cite{ApplegateK91,LV06}, using rapidly-mixing Markov chains whose state space is the body of interest. 

\paragraph{Sampling from the boundary of a convex set.}
This paper discusses the problem of sampling from the {\em boundary} of a convex body.
There are several reasons that warrant a detailed inquiry into such sampling algorithms:

\begin{enumerate}
\item Sampling from the boundary of a convex set generalizes sampling from a convex set $K$, 
since samples from $K$ can be generated by sampling from the boundary of the set $K\times [0,1]$ in $\R^{n+1}$. 
This yields samples from $K\times \{0,1\}$ but also from $\partial K\times [0,1]$, 
and the probability of getting a point in $K\times \{0,1\}$ is at least $2/n$ if $K$ contains a unit ball.



\item There are specific applications for sampling from the boundary of a convex set, see \cite{MR1110353}.
Our own interest is triggered by an application to the theory of ranking and selection \cite{diekerkim2014}, which aims to identify the
best performing system among a group of systems when only noisy observations are available on their performance.
For applications of the more general problem of sampling from manifolds, we refer to 
\cite{MR3124690,diaconisholmesshahshahani:manifold2013,MR2538813}.

\item MCMC algorithms that exploit the boundary could prove to be faster. 
The theoretical results given in this paper do not give better worst-case bounds, but they could lead to an implementation that is faster in practice.
\item Natural Markov chains in this setting can be viewed as stochastic variants of billiards, a classical topic in chaos theory. Thus, stochastic billards could be viewed as a bridge between chaos theory and MCMC algorithms, and there may be some potential for chaos theory to be applied in the context of (approximate) sampling algorithms for convex bodies.
\item New tools need to be developed, which are potentially useful in various other settings. Key hurdles arise due to the fact that we work with curved spaces.
\end{enumerate}

Existing sampling algorithms for convex sets can be adapted to be used as
alternatives to  algorithms that directly sample from the boundary of convex sets.
However, such algorithms are computationally more expensive 
since one needs many samples from the convex body to get one approximate sample from the boundary \cite{MR3124690,MR1616544,MR2538813}.

\paragraph{Stochastic billards.}
Shake-and-bake algorithms \cite{MR1139964,MR1615562} 
have been proposed for sampling from the boundary of a compact set through the MCMC paradigm.
Underlying each of these algorithms is a Markov chain, whose state space is the 
boundary of a convex body and whose stationary distribution is the target distribution 
one seeks samples from. After running the Markov chain for a while, the distribution
of the chain becomes `close' to the target distribution and one thus obtains an approximate
sample from the target distribution.
Clearly, the efficiency of these (and any) MCMC algorithms critically depends on how long the Markov
chain will need to be run to get close to the target distribution.

The stochastic billiard algorithm we study in this paper is a special 
shake-and-bake algorithm (`running shake-and-bake'), and can informally be described as follows.
It traces a ball bouncing inside a set; when the ball hits the boundary of the domain, it is sent in a random direction according to a cosine distribution, depending only on the normal to the tangent at the point of contact with the boundary and {\em not depending} on the incoming direction. 
The Markov chain of hitting points on the boundary has a uniform stationary distribution.
The motivation for this paper is to understand the convergence properties of stochastic billiards, i.e.,
for how many steps this Markov chain has to be run as a function of the body.

The main contribution of this paper is the first (to our knowledge) rapid mixing guarantee for a Markov chain supported by the boundary of a convex body.
For such bodies with bounded curvature, the main theorem implies a polynomial-time algorithm for sampling from a distribution arbitrarily close to uniform on the boundary. We emphasize that the guarantee is polynomial in the dimension for any body in this class. 
This in turn has applications, including estimating the surface area. As we note later, our mixing time bound is asymptotically the best possible.

The worst-case bound derived here is a first theoretical step towards using stochastic billiards for sampling algorithms, 
but a good implementation requires an understanding of several issues of utmost practical significance. 
Chief among these is the question how to sample from a `good' initial distribution, which itself may require running MCMC multiple times with different dynamics.
Another related key question is how to sample from more general densities, without getting trapped by too many rejections of the Metropolis filter.
The recent paper \cite{cousinsvempala} sheds some light on these issues in the context of volume computations.


\paragraph{Related work.}
Hit-and-run is a random walk in a convex body (not its boundary) \cite{Boneh79,MR775260}. Its convergence was first analyzed by Lov\'asz \cite{L2}, who showed that it mixes rapidly from a {\em warm} start, i.e., a distribution close to the stationary. Later, it was shown to be rapidly mixing from any starting point for general logconcave densities \cite{LV3,LV2}. It is similar to stochastic billiards in that each step uses a randomly chosen line through the current point. 
In these and other works on MCMC sampling for convex bodies, the required number of steps for approximate convergence of the Markov chain (`mixing time') only depends on the dimension $n$ of the body $K$ and its diameter $D=\sup\{|x-y| : x,y\in K\}$.

At a high level, our main proof of convergence is similar to previous work. It is based on bounding the {\em conductance} of the Markov chain. This is done via an isoperimetric inequality and an analysis of single steps of the chain (see, e.g., the survey \cite{VemSurvey} on geometric random walks). 
Beyond this high-level outline however, our analysis departs from the standard route and from that of hit-and-run. The isoperimetric inequality we need is for curved spaces, unlike most inequalities in the literature on sampling, which are for convex bodies or logconcave functions. 
The main challenge for proving rapid convergence comes in the analysis of single steps, which is significantly more intricate for stochastic billiards than for hit-and-run.
Here we have to show two things. First that ``proper" steps are substantial and not infrequent; and second that the one-step distributions from two {\em nearby} points have significant overlap. The latter proof has to take into account the specific geometry of stochastic billiards and the cosine law. 

There is a significant body of work on convergence properties of stochastic billiards with a focus 
on establishing geometric ergodicity of the Markov chain, i.e., exponentially fast convergence to the stationary distribution
\cite{%
MR2481068,
MR1843052, 
MR921821} 
on the body, e.g., through the dimension $n$ and diameter $D$ of the body.
It is this dependence that is of paramount importance from an algorithmic point of view,
since this convergence rate determines if the resulting algorithm is polynomial-time or not.

\paragraph{Bounded curvature.}
All mixing time results for sampling from a convex body $K$ rely on the assumption that $K$ is contained in a ball with diameter $D$ and that it contains a ball with diameter $d$. The mixing time then depends on the ratio $D/d$. In this paper, we work under the assumption that the boundary of $K$ has curvature bounded by $\cal C$, which is implies that the body contains a ball with radius $1/{\cal C}$. Thus, in our case, the mixing time depends on the product ${\cal C} D$.

In view of the above, it is natural to ask whether our assumption of bounded curvature could be relaxed to the usual ball-containment condition. 
We find it plausible that this can be done, but we believe this requires new insights on isoperimetric inequalities.
The assumption of bounded curvature allows us to derive a pointwise lower bound on the step size (Lemma~\ref{lem:curvature}), 
which is one of the key ingredients for our mixing time bound. With this pointwise bound, we are able to leverage 
existing (recent) results on isoperimetric inequalities on curved spaces. 

The interest in MCMC algorithms for sampling and volume computations is driven by the fact
that deterministic approximations of the volume of convex bodies are computationally intractable \cite{MR911186}, even under the assumption
of bounded curvature, while 
MCMC yields polynomial-time approximations for volumes of arbitrary convex bodies and measures under logconcave distributions. Such approximations
are also of interest for manifolds, and our work is motivated by sampling and thereby estimating relative measures of boundary subsets. 

\paragraph{Notation.}
Throughout this paper, we study smooth manifolds with a Riemannian metric so that the notions of curvature, Borel sets, and volume are well-defined.
We specifically focus on boundaries of convex bodies. 
We say that the boundary $\partial K$ of a convex body $K$ has {\em curvature bounded from above by ${\cal C}<\infty$} if
for each $x\in\partial K$, there is a ball $B$ with radius $1/{\cal C}$ and center in $K$ so that the tangent planes of $K$ and $B$ at $x$ coincide and $B$ lies in $K$.
For a Borel subset $A\subset \partial K$, we write $\vol(A)$ for its volume (which is actually a surface area).

We write $\Psi$ for the tail of the standard Gaussian distribution, i.e., $\Psi(x)=\int_x^\infty \frac{1}{\sqrt{2\pi}} e^{-y^2/2} dy$, and $\Psi^{-1}$ for its inverse.

The {\em total variation distance} between measures $P$ and $Q$ on $\partial K$ is defined as
\[
\|P-Q\|_\TV = \sup_{A\subseteq \partial K} |P(A)- Q(A)|.
\]
(Here and elsewhere, the necessary measurability assumptions are implicit.)

We write $B^n$ and $S^{n-1}\subseteq \R^n$ for the unit ball and unit sphere in $\R^n$, respectively.
For $x\in\partial K$, let $S_x$ be the unit sphere centered at $x$,
and $n_x$ the inward normal at $x$.
Write $P_x^{\cos}$ for the law with density proportional to $n_x\cdot(y-x)$ on the halfsphere $\{y\in S_x: n_x\cdot (y-x) \ge 0\}$, and $P_x^{\unif}$ for the law with the uniform density on this halfsphere. 
We refer to this law as the cosine law, since the density is proportional to $\cos(\phi_{xy})$, where
where we write $\phi_{xy}$ for the acute angle between $x-y$ and $n_x$, where $x,y\in\partial K$,
see Figure~\ref{fig:SB} below.
We also need two-sided versions:
$\tilde P_x^{\cos}$ is the law on $S_x$ with density proportional to $|n_x\cdot (y-x)|$, and 
$\tilde P_x^{\unif}$ is the uniform distribution on $S_x$.

\section{Stochastic billiards and our main results}
This section describes the stochastic billiards we study in this paper, and presents our main results.

\vspace{3mm}
\begin{mdframed}[style=MyFrame]
{\bf Stochastic billiard on $\partial K$}
\begin{enumerate}
\item Assuming the current state is $x\in\partial K$, sample $w\in S_x$ from the law $P_x^{\cos}$.
\item The next state is the unique intersection point $y\neq x$ of the line $\{x+tw:t\in\R\}$ with $\partial K$.
\item Repeat.
\end{enumerate}
\end{mdframed}

Note that the intersection point always exists and is unique by our assumption on the curvature.
Figure~\ref{fig:SB} illustrates the dynamics.
Efficiently sampling from $P_x^{\cos}$ can be done by first 
sampling uniformly from the $(n-1)$-dimensional unit ball centered at $x$ in the tangent plane at $x$,
and then projecting on $S_x$ in the direction of $n_x$; see \cite{MR1139964,MR1110353}.

\begin{figure}[htp]
\begin{center}
\includegraphics[height=60mm]{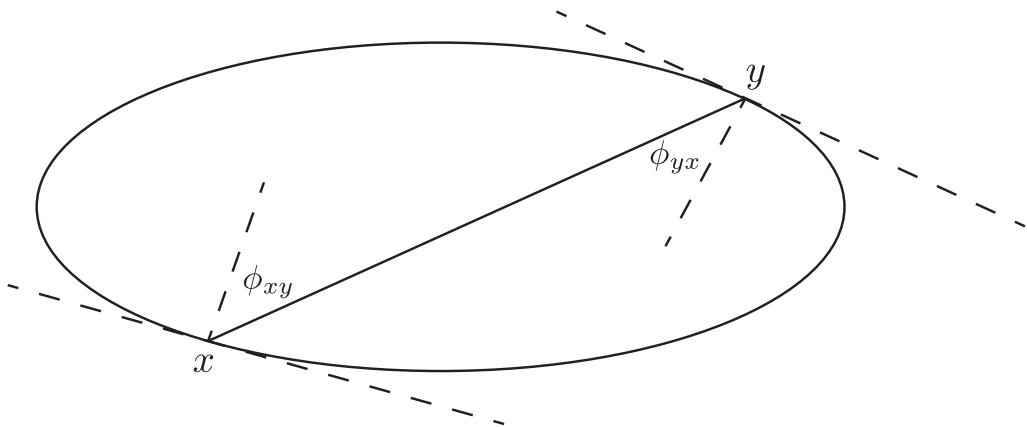}
\end{center}
\caption{The stochastic billiard moves from $x$ to $y$. The inward normals and the tangent planes at $x$ and $y$ are dashed.}
\label{fig:SB}
\end{figure}

The subsequent intersection points $\{X_k\}$ form a random sequence on $\partial K$, and it is immediate that this is a Markov chain. We call it the {\em stochastic billiard Markov chain}.
The one-step distribution of the Markov chain is, for $u\in\partial K$ and $A\subseteq \partial K$, \cite{MR1139964,MR1615562}
 \begin{equation}\label{one-step-formula}
P_u(A) = \frac{\pi^{(n-1)/2}}{\vol(\partial K) \Gamma((n+1)/2)}\int_A \frac{\cos(\phi_{uv})\cos(\phi_{vu})}{\|u-v\|^{n-1}} dv,
\end{equation}
where $\Gamma$ is the gamma function.

It is worthwhile to understand why the cosine distribution is a natural choice for
the outgoing direction in the stochastic billiard chain, since this is closely related
to the fact that the uniform distribution is stationary for the chain.
As can be seen in Figure~\ref{fig:incidence_angle},
the more oblique the incidence of a bundle, its mass must be `spread out' over a larger region.
It is readily seen that this effect is proportional to the cosine of the incoming angle $\phi$ (as defined previously with respect to the normal).
As a result, the two cosines appearing in the transition density make the kernel symmetric and consequently
the uniform distribution is stationary.
Thus, the next lemma forms the starting point of MCMC algorithms 
for sampling from the uniform distribution $\pi$ on $\partial K$, 
We refer to \cite{MR1139964,MR1110353} for proofs of this lemma.
\begin{figure}[htp]
\begin{center}
\includegraphics[height=50mm]{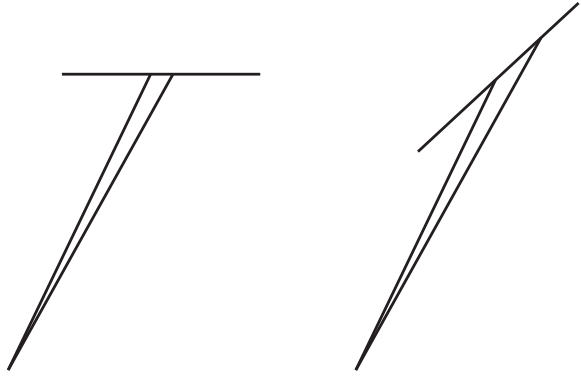}
\end{center}
\caption{The same bundle with different incidence angles.}
\label{fig:incidence_angle}
\end{figure}

\begin{lemma}
The uniform distribution $\pi$ on $\partial K$ is stationary for the stochastic billiard Markov chain.
Moreover, for any initial state $x\in\partial K$, we have
\[
\lim_{k\to\infty} P(X_k\in B|X_0=x) = \pi(B).
\]
\end{lemma}

To use the stochastic billiard chain as a MCMC sampler for approximate sampling from $\pi$, the chain 
must be stopped after an appropriate number of steps.
Our main result is that it takes order ${{\cal C}^2n^2D^2}$ steps to get
arbitrary close to its uniform equilibrium distribution if the initial distribution is `good', where ${\cal C}$ is the upper bound on the 
curvature of $\partial K$ and 
$D$ is the diameter of $K$, i.e., the largest distance between any two points in the body.
A different way of phrasing this result is to say that the mixing time of the stochastic billiard
chain is order ${{\cal C}^2n^2D^2}$.
 This is a uniform bound over all bodies with diameter at most $D$ and curvature bound ${\cal C}$, 
and it is asymptotic in the dimension $n$.

\begin{theorem}
\label{thm:main}
Let $K$ be a convex body in $\R^n$ with diameter $D$. Suppose 
that the curvature of $\partial K$ is bounded from above by ${\cal C}$. 
Set $M=\sup_A Q_0(A)/\pi(A)$. Then there is a constant $c$ such that, for $k\ge 0$,
\[
\|Q_k-\pi\|_\TV \le \sqrt M \left(1-\frac {c}{{\cal C}^2n^2D^2}\right)^k.
\]
\end{theorem}

We remark that an explicit expression for the constant $c$ can be found by tracing 
constants in the proofs of this paper.
Since it is not our objective to find the sharpest possible constant,
we do not specify the constant in this statement.

We prove Theorem~\ref{thm:main} by bounding the so-called conductance of the stochastic billiard chain;
this is a standard tool for bounding rates of convergence to stationarity for Markov chains.
The conductance is defined as
\[
\phi = \inf_{A\subseteq \partial K: 0<\pi(A)<1/2} \frac{\int_A P_u(\partial K\backslash A) d\pi(u)}{\pi(A)}.
\]
The next proposition states our main result on the conductance of the stochastic billiard chain.
Theorem~\ref{thm:main} immediately follows from this proposition in conjunction with 
Corollary~1.5 from \cite{lovaszsimonovits:1993}.
The constant $c$ is different from the one in Theorem~\ref{thm:main}.

\begin{prop}
\label{prop:condSB}
Under the assumptions of Theorem~\ref{thm:main},
the conductance $\phi$ of the stochastic billiard chain $\{X_k\}$ satisfies 
\[
\phi\ge \frac{c}{{\cal C} nD},
\] 
for some universal constant $c$.
\end{prop}

The following proposition shows that the mixing time bound of $O({\cal C}^2n^2D^2)$ cannot be improved unless
further conditions are imposed.
This proposition is similar to the lower bound for hit-and-run shown in \cite[Sec.~8]{L2}, but we use different arguments.

\begin{prop}
Consider the stochastic billiard process on $\partial K$ with 
\[
K = \{x\in \R^n: |x_1|\le D/2, x_2^2+\cdots+x_n^2\le 1/{\cal C}^2\}$ and $D/2\ge 1/{\cal C}.
\]
For $\epsilon>0$ small enough, we have 
\[
\| Q_{\epsilon  {\cal C}^2 n^2 D^2} -\pi\|_\TV \ge 1/16.
\]
\end{prop}
\begin{proof}
We may assume without loss of generality that $\mathcal C=1$ by scaling $K$ if this were not the case (which has no impact on the dynamics of the stochastic billard).
We study the first coordinate $\{X_{k,1}:k\ge0\}$ 
of the stochastic billiard process on the boundary of $\R \times\{0\}^{n-1} + B_n$. Due to symmetry, this process is a centered random walk on $\R$.

The first step of the proof is to study the step size distribution of this random walk. To this end, it is convenient to shift the body so that the origin lies on the boundary.
We thus consider $\tilde K=\R \times\{0\}^{n-1} + B_n - (0,-1,0,\ldots,0)$, which can also be written as
$\tilde K=\{x\in \R^n: (1-x_2)^2+x_3^2 +\ldots+x_n^2\le 1\}$.

Let $X$ have the uniform distribution on an $(n-1)$-dimensional unit ball tangent to $\tilde K$ at the origin; the tangent plane is 
orthogonal to $(0,1,0,\ldots,0)$.
Let $Y$ be the projection of $X$ on the unit halfsphere $B_n\cap \{x:x_2\ge 0\}$, i.e.,
\[
Y = (X_1, \sqrt{1-\|X\|^2}, X_3, \ldots, X_n).
\]
It is readily verified that the line between the origin and $Y$ intersects $\partial {\tilde K}$ at the point
\[
Z=\frac{2Y}{1-X_1^2} \sqrt{1-\|X\|^2}.
\]
Note that the distribution of $Z$ equals the one-step distribution of the stochastic billiard process starting from the origin.

We next show that $\Var(Z_1)$ is of order $1/n^2$. We note that
\[
P(\|X\|^2-X_1^2\le r, |X_1|\le x) = 2xr^{n-2} \frac{\vol(B_{n-2})}{\vol(B_{n-1})},
\]
so for $r^2+x^2\le 1, r\ge 0$ the density of $(\|X\|^2-X_1^2, X_1^2)$ is
\[
f(r,x) = \frac{\vol(B_{n-1})}{\vol(B_n)} (n-2) r^{n-3},
\]
which shows that
\[
\Var(Z_1) = \E \left[\left(1-\|X\|^2\right)\frac{4X_1^2}{(1-X_1^2)^2}\right] = \int_{r\ge 0, r^2+x^2\le 1} (1-r^2-x^2) \frac{4x^2}{(1-x^2)^2} f(r,x) dxdr,
\]
and it is readily verified that this is of order $1/n^2$.

Thus, in the $x_1$ direction, the stochastic billiard process on $\tilde K$ is a centered random walk and its step size has standard deviation of order $1/n$. 
Writing $\tau_x = \inf\{k\ge 0: X_{k,1}\ge x\}$ for $x>0$, we know that $\tau_{D/4}$ is of order $n^2D^2$ as
can for instance be deduced from Donsker's Central Limit Theorem for random walks.

Since $D\ge 2$, we have
\[
\pi(\{x:x_1\ge D/4\}) \ge \frac{D/4}{D+2} \ge \frac 18.
\]
Upon comparing the stochastic billiard process on the boundary of $\tilde K$ versus $K$, we obtain that
\[
Q_{\epsilon  n^2 D^2} (\{x:x_1\ge D/4\}) \le P(\tau_{D/4} \le \epsilon  n^2 D^2),
\]
which is at most $1/16$ for small enough $\epsilon$.
We conclude that
\[
\| Q_{\epsilon  n^2 D^2} - \pi\|_\TV \ge \pi(\{x:x_1\ge D/4\})  - Q_{\epsilon  n^2 D^2} (\{x:x_1\ge D/4\}) \ge 1/16,
\]
and this establishes the claim.
\end{proof}

\section{Single step analysis}

This section studies the one-step distribution in detail.
For $x\in \partial K$, define $F(x)$ through
\[
P_x(y\in\partial K: |x-y|\le F(x)) = \frac 1{128}.
\]
We think of $F(x)$ as the `median' step size from $x$.
The goals of this section are two-fold: (1) to establish a pointwise lower bound on $F$ which does not depend on $x$
and (2) to establish that the distribution of two points on the boundary $\partial K$ overlap considerably if the points are sufficiently close.
We discuss these two parts in Sections~\ref{sec:LBonF} and \ref{sec:overlap}, respectively.

The proofs of these two parts proceed essentially independently, but the following auxiliary lemma is used in both parts.
The lemma intuitively says that the projection of $Y$ onto the normal at $x$ lies at distance about $1/\sqrt{n}$ from $x$ if the distribution of $Y$ is
$P^{\cos}_x$ or $P^{\unif}_x$, and gives precise asymptotic characterizations of the asymptotic probabilities.

\begin{lemma}
\label{lem:volumelemma}
For $A\subseteq \R_+$, we have
\begin{eqnarray*}
\lim_{n\to\infty} P^{\cos}_x ( \{y\in S_x: n_x\cdot (y-x) \in A/\sqrt{n}\}) &= &
\int_{A} x \exp(-x^2/2) dx, \\
\lim_{n\to\infty} P^{\unif}_x ( \{y\in S_x: n_x\cdot (y-x) \in A/\sqrt{n}\}) &=& 
\int_{A} \frac{1}{\sqrt{\pi/2}} \exp(-x^2/2) dx,
\end{eqnarray*}
and for $A\subseteq \R$, we have
\begin{eqnarray*}
\lim_{n\to\infty} \tilde P^{\cos}_x ( \{y\in S_x: n_x\cdot (y-x) \in A/\sqrt{n}\}) &=& 
\int_{A} \frac{x}{2} \exp(-x^2/2) dx, \\
\lim_{n\to\infty} \tilde P^{\unif}_x ( \{y\in S_x: n_x\cdot (y-x) \in A/\sqrt{n}\}) &=& 
\int_{A} \frac{1}{\sqrt{2\pi}} \exp(-x^2/2) dx.
\end{eqnarray*}
\end{lemma}
\begin{proof}
The key ingredient in the proof is the observation that an $(n-2)$-dimensional sphere with radius $r$
has surface measure proportional to $r^{n-2}$. The radius of the 
$(n-2)$-dimensional sphere $\{y\in S_x: n_x\cdot (y-x)=t\}$ is $\sqrt{1-t^2}$, so for the probability
under $P^{\cos}_x$ we find that
\[
\lim_{n\to\infty} \frac{\int_{A/\sqrt{n}}  t (\sqrt{1-t^2})^{n-2} dt} 
{\int_{0}^{\infty}  t (\sqrt{1-t^2})^{n-2} dt} = \frac{\int_A s\exp(-s^2/2)ds}{\int_0^\infty s\exp(-s^2/2)ds}
=\int_A s\exp(-s^2/2)ds.
\]
Similarly, for the probability under $P^{\unif}_x$ we find that
\[
\lim_{n\to\infty} \frac{\int_{A/\sqrt{n}}  (\sqrt{1-t^2})^{n-2} dt} 
{\int_{0}^{\infty}  (\sqrt{1-t^2})^{n-2} dt} = \frac{\int_A \exp(-s^2/2)ds}{\int_0^\infty \exp(-s^2/2)ds}
=\int_A \frac{1}{\sqrt{\pi/2}}\exp(-s^2/2)ds.
\]
The statements for $\tilde P^{\cos}_x$ and $\tilde P^{\unif}_x$ are found analogously.
\end{proof}

\subsection{Lower bound on $F$}
\label{sec:LBonF}

The main result of this subsection is the following lemma, 
which guarantees that the stochastic billiard Markov chain 
makes `large enough' steps with good probability.

\begin{lemma}
\label{lem:LBonF}
If $K$ is a convex body with curvature bounded from above by ${\cal C} < \infty$,
then there exists a constant $c$ such that  $F(x)\ge c/({\cal C} \sqrt{n})$ for $x\in\partial K$.
\end{lemma}

We prove this lemma by comparing $F$ with a family of functions $\{s_\gamma:0\le\gamma\le 1\}$ 
that is easier to bound.
As with $F$, one similarly interprets $s_\gamma$ as a step size.
For $1 \ge \gamma \ge 0$ and $x \in K$, it is defined through
\begin{equation}
\label{eq:defs}
s_\gamma(x) = \sup\left\{ t\ge0 \, : \, \frac{\vol((x+ tB^n) \cap K)}{\vol(tB^n)} \ge \gamma\right\}.
\end{equation}
We are now ready to formulate our comparison between $F$ and $s_\gamma$.

\begin{lemma}\label{lem:F-s}
For any $u \in \partial K$ and $1/2-1/2048 \le \gamma \le 1/2$,
we have for large enough $n$,
\[
F(u) \ge \frac {s_\gamma(u)}{2}.
\]
\end{lemma}
\begin{proof}
Fix $u\in \partial K$, and write $s = s_{\gamma}(u)$.
Let $T\subseteq S_u$ be given by
\[
T=\{u+ y: u+(s/2) y \in K, y\in S^{n-1}\}.
\]
Writing $p=\tilde P_u^{\unif}(S_u\backslash T)$ for the fraction of $u+(s/2)S^{n-1}$ that lies outside $K$, we find 
by convexity of $K$ that
\[
\vol((u + s B^n) \setminus K) \ge p \vol(sB^n)-\vol\left(\frac s2 B^n\right).
\]
By definition of $s$, we also have $\vol((u+sB^n) \setminus K) \le (1-\gamma)\vol(sB^n)$.
Since $\vol(\alpha B^n) = \alpha^n\vol(B^n)$, we deduce that for large enough $n$ 
and $\gamma \ge 1023/2048$ that
\[
p \le 1-\gamma + \frac{1}{2^n} \le \frac{513}{1024}.
\]
Note that, by definition of $p$, 
\[
P_u^{\unif}(S_u\backslash T) =1- 2(1-p)\le 1/512.
\]
We next argue that $P^{\cos}_u(S_u\backslash T)\le 1/128$, 
so that $F(u)\ge s/2$ then follows from 
\[
P_u\left(y\in\partial K: |u-y| < \frac s2\right)  = 1-P_u^{\cos}(T) \le \frac{1}{128}.
\]

To show that $P^{\cos}_u(S_u\backslash T)\le 1/128$, we use a change-of-measure argument.
Let $f^{\cos}$ and $f^{\unif}$ be the densities on $S_u$ of $P^{\cos}$ and $P^{\unif}$, respectively,
and write $E_u^{\unif}$ for the expectation operator of $P^{\unif}$.
We let $X$ denote the random variable given by $X(\omega) =\omega$ for $\omega\in S_u$.
After noting that
\[
\frac{f^{\cos}(y)}{f^{\unif}(y)}=\frac{\cos(\phi_{uy})}{E_u^{\unif} [n_u\cdot (X-u)]},
\]
we write
\begin{eqnarray*}
P_u^{\cos} (S_u\backslash T) &=& E_u^{\unif} \left[\frac{f^{\cos}(X)}{f^{\unif}(X)}; X\in S_u\backslash T\right] \\
&\le& \sup_{A: P_u^{\unif}(A)\le 1/512} E_u^{\unif} \left[\frac{f^{\cos}(X)}{f^{\unif}(X)}; X\in A\right] \\
&=&  \sup_{A: P_u^{\unif}(A)\le 1/512} \frac{E_u^{\unif} [\cos(\phi_{uX}); X\in A]}{E_u^{\unif} [n_u\cdot (X-u)]} \\
&=& \frac{E_u^{\unif} [\cos(\phi_{uX}); X\in C^{1/512}]}{E_u^{\unif} [n_u\cdot (X-u)]},
\end{eqnarray*}
where $C^{1/512}=\{x\in S_u: n_u\cdot (x-u)\ge c'_n/\sqrt{n}\}$ is the cap of $S_u$ in `direction' $n_u$,
where $c'_n$ is chosen so that $P_u^{\unif}(C^{1/512}) = 1/512$.
Note that $c'_n\to \Psi^{-1}(1/1024)$ as $n\to\infty$ by Lemma~\ref{lem:volumelemma}.
The asymptotic behavior as $n\to\infty$ of the denominator is readily found:
\[
E_u^{\unif} [n_u\cdot (X-u)] = \int_0^1 \sqrt{1-r^2} \, (n-1) r^{n-2} dr =
\frac{\sqrt{\pi} (n-1) \Gamma((n-1)/2)} {4 \Gamma(n/2+1)}
\sim \sqrt{\frac{\pi}{2n}},
\]
where the notation $f(n)\sim g(n)$ as $n\to\infty$ is shorthand for $\lim_{n\to\infty} f(n)/g(n)=1$.
As for the numerator, we find by again applying Lemma~\ref{lem:volumelemma} that
\[
E_u^{\unif} [\cos(\phi_{uX}); X\in C^{1/512}] \sim \frac{1}{\sqrt{n}} \int_{\Psi^{-1}(1/1024)}^\infty s \exp(-s^2/2) ds = \frac{\exp(-\Psi^{-1}(1/1024)^2/2)}{\sqrt{n}}.
\]
Since $\sqrt{2/\pi}\exp(-\Psi^{-1}(1/1024)^2/2) < 1/128$, we find that for large $n$,
\[
P_u^{\cos}(S_u\backslash T) \le 1/128,
\]
which proves the claim.
\end{proof}

Lemma~\ref{lem:LBonF} follows by combining the preceding lemma with the following result, which establishes a lower bound on the `step size' $s_\gamma$
as defined in (\ref{eq:defs}).

\begin{lemma}
\label{lem:curvature}
Let $\gamma\in(0,1/2)$.
If $K$ is a convex body with curvature bounded from above by ${\cal C}<\infty$, then 
$$s_\gamma(x)\ge \frac {c_\gamma}{{\cal C}\sqrt{n}},$$
where $c_\gamma$ is a constant only depending on $\gamma$.
\end{lemma}
\begin{proof}
Fix $x\in\partial K$. In view of the definition of $s_\gamma$, it suffices to show that 
\[
\vol((x+t B)\cap K) \ge \gamma \vol(t B),
\]
for $t=c_\gamma /({\cal C}\sqrt{n})$.

Let $B^{\cal C}$ be a ball with radius $1/{\cal C}$ and center in $K$ so that the tangent planes of $K$ and $B^{\cal C}$ at $x$ coincide and $B^{\cal C}$ lies in $K$.
Let $B_t$ be a ball with radius $t<1/{\cal C}$ centered at $x$.
Let $x$ be the origin of a new coordinate system, in which the center of the larger ball is $(1/{\cal C},0,\ldots,0)$.

We first characterize the points (in the new coordinate system) where the boundaries of the two balls intersect.
All of these points have the same first coordinate, namely equal to $y$ satisfying
\[
t^2 - y^2 = \frac1{{\cal C}^2} - \left(\frac 1{\cal C} - y\right)^2,
\]
and the solution is $y={\cal C} t^2/2$, and we write $H_{{\cal C} t^2/2}$ for the halfspace of all points with 
first coordinate exceeding $y$.

We now use the property that, for a unit ball $B$, the volume of all points with first coordinate exceeding $1/\sqrt{n}$ takes up a constant fraction of its volume.
Consequently, if $t=c_\gamma/({\cal C}\sqrt{n})$ for an appropriate constant $c_\gamma$, 
then ${\cal C} t^2/2 = c_\gamma t/(2\sqrt{n})$ and therefore 
$\vol ( B_{t}\cap H_{{\cal C} t^2/2}) = \gamma \vol(B_{t})$.
Upon noting that, for this choice of the radius $t$, 
\[
\vol((x+tB)\cap K) \ge \vol ( B_{t}\cap H_{{\cal C} t^2/2}) = \gamma \vol(B_t) =\gamma \vol(tB),
\]
we obtain the claim.
\end{proof}

\subsection{Overlap for points that are close}
\label{sec:overlap}

It is the aim of this subsection to prove the following lemma, 
which states that the transition probabilities for points that are sufficiently close must 
be similar. This is formalized as having (total variation)
`overlap' at least $1/\kappa>0$, for some $\kappa$.

\begin{lemma}
\label{lem:overlapSB}
Let $u,v\in \partial K$, and let $n$ be large. If
\[
|u-v|< \frac{1}{100\sqrt n}\max(F(u),F(v)),
\]
then we have
\[
\|P_u-P_v\|_{\text{TV}}\le 1-\frac 1{\kappa},
\]
where $\kappa$ is an absolute constant determined in the proof.
\end{lemma}
\begin{proof}
Let $u,v \in \partial K$ be as in the hypothesis of the lemma with $F(u) \ge F(v)$.
Our main idea is to compare the transition densities of from $u$ and $v$ on a set of full measure.
We first introduce four subsets $A_1,\ldots,A_4$ of $\partial K$ which we exclude from this comparison.

{\bf The set $A_1$.}
Let $A_1$ be the subset of $\partial K$ that is close to $u$, i.e., 
\[
A_1 = \{ x \in \partial K \, : \, |u-x| \le F(u) \}.
\]
Note that $P_u(A_1) = 1/128$ by definition of $F(u)$.

{\bf The set $A_2$.}
Next we define $A_2$ as the subset of points that are far from being orthogonal to $[u,v]$, the line through $u$ and $v$ (interpreting $u$ as the origin):
\[
A_2 = \left\{x \in \partial K \, : \, |(x-u)^T(v-u)| \ge \frac{3}{\sqrt{n}}|x-u||v-u|\right\}.
\]
\begin{claim}
For large enough $n$, we have
$P_u(A_2)\le 1/64$.  
\end{claim}
To prove this claim, consider the two-dimensional plane consisting of $u$, $v$, and $n_u$. 
Recall that $S_u$ stands for the unit sphere centered at $u$.
Define
$C_y=\{x\in S_u: |y\cdot  (x-u)|/|y| \ge 3/\sqrt{n}\}$ as the union of two caps centered around 
the line determined by $y$.
Note that $x\in C_{v}$ if and only if either $x\in A_2$ or $-x\in A_2$, and that the latter are mutually exclusive statements.
Thus we have $P_u(A_2) = \tilde P^{\cos}_u(C_v) \le \sup_{y\in S_u} \tilde P^{\cos}_u(C_y)$, and the supremum is attained for $y=n_u$ 
due to the form of the density of $\tilde P^{\cos}_u$.
Therefore, as $n\to\infty$, we have by Lemma~\ref{lem:volumelemma} that
\[
P_u(A_2)=\tilde P^{\cos}_u(C_{n_u}) \to \int_3^\infty s\exp(-s^2/2) ds = \exp(-3^2/2).
\]
The right-hand side is less than $1/64$.

{\bf The set $A_3$.}
Let $A_3$ be given by
\[
A_3 = \left\{x \in \partial K\, : \sqrt{n} \cos(\phi_{ux}) \not \in (c_1,c_2)\right\}\cup
\left\{x \in \partial K\, : \sqrt{n} \cos(\phi_{vx}) \not \in (c_1,c_2)\right\},
\]
where $c_1$ and $c_2$ satisfy 
\begin{eqnarray*}
\lim_{n\to\infty} P_u(x \in \partial K\, : \sqrt{n} \cos(\phi_{ux}) <c_1) &=& 1/64 \\
\lim_{n\to\infty} P_u(x \in \partial K\, : \sqrt{n} \cos(\phi_{ux}) >c_2) &=& 1/64.
\end{eqnarray*}
Note that the limits on the left-hand side are equal to $1-\exp(-c_1^2/2)$ and
$\exp(-c_2^2/2)$, respectively, by Lemma~\ref{lem:volumelemma}.
We can thus set 
$c_1 = \sqrt{-2\log(63/64)}\approx 0.18$ and $c_2= \sqrt{-2\log(1/64)} \approx 2.88$.

\begin{claim}
\label{claim:badoutgoingangles}
For large enough $n$, we have $P_u(A_1^c\cap A_3)\le 30/64$  
and therefore $P_u(A_1\cup A_3) \le 61/128$.
\end{claim}
By definition of $c_1$ and $c_2$, it suffices to show that $P_u(x\in A_1^c: \sqrt{n}\cos(\phi_{vx})\not\in(c_1,c_2))\le 28/64$.
We first show that $P_u(x\in A_1^c: \sqrt{n}\cos(\phi_{vx})>c_2) \le \exp(-(c_2-1/100)^2/2)$.
Since $\cos(\phi_{vx}) = n_v\cdot (x-v)/|x-v|$, we have to bound $P_u(x\in A_1^c: n_v\cdot (x-v)>c_2|x-v|/\sqrt{n})$.
The key ingredient is the following observation: if $n_v\cdot  (x-v)\ge c_2|x-v|/\sqrt{n}$ and $|x-u|> F(u)$, then 
\[
n_v\cdot (x-u) \ge \frac{c_2}{\sqrt{n}}|x-v| + n_v\cdot (v-u) \ge \frac{c_2}{\sqrt{n}}|x-u| - \left(1+\frac{c_2}{\sqrt{n}}\right)|v-u|
> \left(\frac{c_2-1/100}{\sqrt{n}} - \frac{c_2}{100n}\right)|x-u|,
\]
where the last inequality uses $|v-u|< F(u)/(100\sqrt{n})\le |x-u|/(100\sqrt{n})$, which holds since $x\in A_1^c$.
We thus deduce that, for large enough $n$,
\begin{equation}
\label{eq:boundA1cA3}
P_u(x\in A_1^c: n_v\cdot (x-v)>c_2|x-v|/\sqrt{n})\le P_u\left(x\in \partial K: n_v\cdot (x-u) > \left(\frac{c_2-1/50}{\sqrt{n}}\right)|x-u|\right).
\end{equation}
We now bound this probability.
For a unit vector $y$, write $C'_y = \{x\in S_u: y\cdot (x-u)\ge (c_2-1/50)/\sqrt{n}\}$, which is a cap of $S_u$ with `center' $y$.
The right-hand side of (\ref{eq:boundA1cA3}) equals $P^{\cos}_u(C'_{n_v})$. Due to the form of the density of $P^{\cos}_u$, we have
$P^{\cos}_u(C'_{n_v}) \le \sup_{y\in S_u} P^{\cos}_u(C'_{y}) = P^{\cos}_u(C'_{n_u})\to \exp(-(c_2-1/50)^2/2)$.

We next bound $P_u(x\in A_1^c: \sqrt{n}\cos(\phi_{vx})<c_1)$, which is equal to
\[
P_u(x\in A_1^c: 0\le n_v\cdot (x-v)\le c_1|x-v|/\sqrt{n}).
\]
We use a similar argument as before.
If $x\in A_1^c$ and $0\le n_v\cdot (x-v)\le c_1|x-v|/\sqrt{n}$, we have
\[
n_v\cdot (x-u) \ge n_v\cdot (v-u) \ge -|v-u|> -|x-u|/(100\sqrt{n})
\]
and
\[
n_v\cdot (x-u) \le \frac{c_1}{\sqrt{n}} |x-v| + n_v\cdot (v-u) \le \frac{c_1}{\sqrt{n}}|x-u| + \left(1+\frac{c_1}{\sqrt{n}}\right)|v-u|
< \left(\frac{c_1+1/50}{\sqrt{n}}\right)|x-u|.
\]
For a unit vector $y$, write $C''_y = \{x\in S_u: -1/(100\sqrt{n})\le y\cdot (x-u)\le (c_1+1/50)/\sqrt{n}\}$.
We have now shown that 
\[
P_u(x\in A_1^c: \sqrt{n}\cos(\phi_{vx})<c_1)\le P^{\cos}_u(C''_{n_v}).
\]
Note that by Lemma~\ref{lem:volumelemma}, we have
\[
\tilde P^{\unif}_u (C''_{n_v}) \to \frac{1}{\sqrt{2\pi}} {\int_{-1/100}^{c_1+1/50} \exp(-y^2/2) dy}{}.
\]
Call the ratio on the right-hand side $\rho$, and we find that $\rho\approx 0.082$.
Application of Lemma~\ref{lem:volumelemma} (twice) yields
\[
P^{\cos}_u (C''_{n_v}) \le \sup_{D:\tilde P^{\unif}_u(D)=\rho} P^{\cos}_u(D) =P^{\cos}_u (\{x\in S_u: n_u\cdot  (x-u)\ge \Psi^{-1}(\rho)/\sqrt{n}\}) = \exp(-\Psi^{-1}(\rho)^2/2),
\]
where $\Psi(x) =\int_x^\infty \exp(-y^2/2)dy/\sqrt{2\pi}$.
We conclude that 
\[
\limsup_{n\to\infty} P_u(A_1^c\cap A_3) \le \frac{2}{64}+\exp(-(c_2-1/50)^2/2) + \exp(-\Psi^{-1}(\rho)^2/2).
\]
It is readily verified that the right-hand side approximately equals $0.40$, and the first part of Claim~\ref{claim:badoutgoingangles} follows.
For the second part, note that $P_u(A_1\cup A_3) \le P(A_1) + P(A_1^c\cap A_3)$.

{\bf The set $A_4$.}
For the following argument, we interpret $u$ as the origin of our coordinate system, 
so that (for instance) cones are defined with respect to $u$.
For $x\in \partial K$, let $\mathcal C(x)$ be the cone generated by the orthogonal projection of $x$ on
the hyperplane $\{z: n_u\cdot (z-u) = 0\}$ and the normal $u+n_u$ at $u$.
Write $\xi(x)$ for
the angle between the point in $\mathcal C(x)\cap (u+F(u) B^n) \cap \partial K$ 
and the aforementioned hyperplane, see Figure~\ref{fig:overlap}.
\begin{figure}
\begin{center}
\includegraphics[height=60mm]{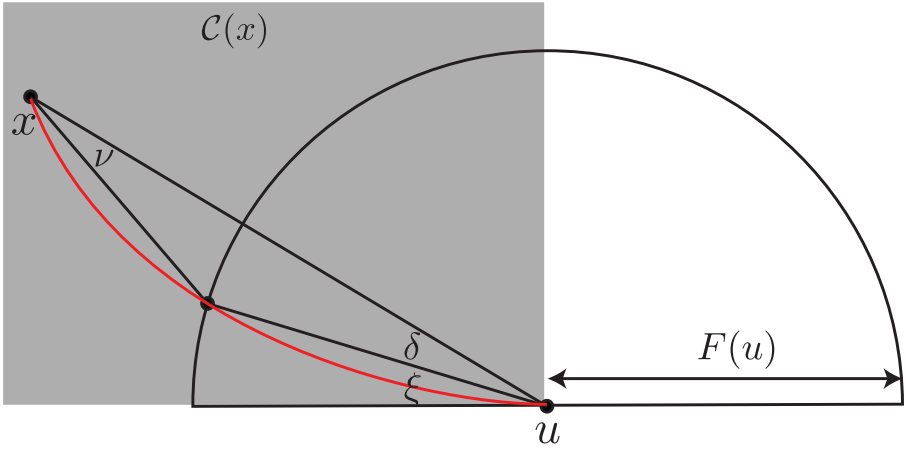}
\end{center}
\caption{Incidence angles and the cone $\mathcal C(x)$. 
Part of the boundary $\partial K$ is depicted in red.}
\label{fig:overlap}
\end{figure}
Write 
\[
A_4 = \{x\in\partial K:\xi(x)\ge c_1/\sqrt{n}\}
\]
and $B = \{ x\in\partial K: \sqrt{n} \cos(\phi_{ux})<c_1\}$. 
Since $P_u(B) = 1/64$ and $P_u(A_1) = 1/128$, we find that 
\[
P_u(A_1^c|B) = 1 - P_u(A_1\cap B)/P_u(B) \ge 1-P_u(A_1)/P_u(B) = 1/2. 
\]
(The last equality only holds asymptotically as $n\to\infty$, but 
we ignore such issues in the remainder of this proof for purposes of readability.)

The angle $\xi(x)$ is only a function of $x$ through $\mathcal C(x)$, i.e.,
$\xi(x)$ is determined once $\mathcal C(x)$ is given.
Interpreting $\mathcal C$ and $\xi$ as random variables on the sample space $\partial K$,
the distribution of $\mathcal C$ under $P_u$
is the uniform distribution over all such cones. 
Since the distribution of $\mathcal C$ under $P_u(\cdot|B)$
is also uniform, the distribution of $\xi$ is the same under $P_u$ and under $P_u(\cdot|B)$.
We conclude that
\[
P_u(A_4^c) = P_u(\xi<c_1/\sqrt{n})= P_u(\xi<c_1/\sqrt{n} | B) \ge P_u (A_1^c|B)\ge 1/2,
\]
so that $P_u(A_4) \le 1/2$.

{\bf A set on which $P_v$ majorizes $P_v$.}
Let $A = \partial K \setminus A_1 \setminus A_2 \setminus A_3\setminus A_4$. Then we have
\[
P_u(A) \ge 1 - \frac 1{128} - \frac 1{64} -\frac{30}{64}-\frac 12= \frac{1}{128}. 
\]

We will show that for any subset $S \subseteq A$, we have
\begin{equation}
\label{eq:PvgePu}
P_v(S) \ge \frac{128P_u(S)}{\kappa},
\end{equation}
where $\kappa$ is some positive constant.
This implies that for any subset $S$ of $\partial K$, we have
\begin{eqnarray*}
P_u(S) - P_v(S) &\le& P_u(S) - P_v(S\setminus A_1\setminus A_2\setminus A_3\setminus A_4) \\
&\le& P_u(S) - \frac {128}{\kappa} P_u(S\setminus A_1\setminus A_2\setminus A_3\setminus A_4) \\
&\le& P_u(S) - \frac {128}{\kappa} [P_u(S)-P_u(A_1\cup A_2\cup A_3\cup A_4)] \\
&\le& P_u(S) - \frac {128}{\kappa} [P_u(S)-127/128] \\
&\le& 1-\frac{1}{\kappa},
\end{eqnarray*}
and therefore we obtain the conclusion of the lemma from (\ref{eq:PvgePu}).  

We prove (\ref{eq:PvgePu}) using the formula for the one-step distribution from $v$ as given in (\ref{one-step-formula}):
\[
P_v(S) = \frac{\pi^{(n-1)/2}}{\vol(\partial K) \Gamma((n+1)/2)}\int_S \frac{\cos(\phi_{vx})\cos(\phi_{xv})}{|v-x|^{n-1}} dx,
\]
and we compare the three terms in the integrand with the corresponding quantities for $v$ replaced with $u$.
This rests on the following three claims for $x \in A$, which show that we can take $\kappa>0$ to satisfy
\[
\frac {128}{\kappa} = e^{-7/2} \frac{c_1}{c_2} \left(1-\frac1{100(c_2-c_1)}\right).
\]
\begin{claim}\label{claim:distance} 
For $x\in A$, we have 
\[
|v-x| \le \left(1+ \frac{7}{2n}\right)|u-x|.
\] 
\end{claim}

To prove Claim~\ref{claim:distance}, we note that for $x \in A$, 
\[
|x-u| \ge F(u) \ge \sqrt{n}|u-v|
\]
and 
\[
|(x-u)^T(v-u)| \le \frac{3}{\sqrt{n}}|x-u||v-u|.
\]
Using these, we deduce that
\begin{eqnarray*}
|x-v|^2 &=& |x-u|^2 + |u-v|^2 + 2(x-u)^T(u-v) \\
&\le& |x-u|^2 + |u-v|^2 + \frac{6}{\sqrt{n}}|x-u||u-v|\\
&\le& |x-u|^2 + \frac{1}{n}|x-u|^2 + \frac{6}{n}|x-u|^2\\
&\le& \left(1+\frac 7n\right)|x-u|^2,
\end{eqnarray*}
which completes the proof of Claim~\ref{claim:distance}.

\begin{claim}\label{claim:phi-vx}
For $n$ large enough and $x\in A$, we have 
\[
\frac{\cos(\phi_{vx})}{\cos(\phi_{ux})} \ge \frac{c_1}{c_2}.
\]
\end{claim}

Claim~\ref{claim:phi-vx} immediately follows upon noting that for $x\in A$, since $x\not\in A_3$,
$\sqrt{n} \cos(\phi_{vx}) \ge c_1$ and $\sqrt{n} \cos(\phi_{ux}) \le c_2$; 
therefore $\cos(\phi_{ux})$ and $\cos(\phi_{vx})$ are within a factor of $c_2/c_1$. 

\begin{claim}\label{claim:phi-xv}
For $n$ large enough and $x\in A$, we have 
\[
\frac{\cos(\phi_{xv})}{\cos(\phi_{xu})} \ge 1-\frac 1{100(c_2-c_1)}.
\]
\end{claim}

To prove Claim~\ref{claim:phi-xv}, we need to derive a lower bound on
$\cos(\phi_{xv})/\cos(\phi_{xu})$. 
Fixing $u$, $\cos(\phi_{xv})/\cos(\phi_{xu})$ achieves its lowest possible value when $v$ lies in $\mathcal C(x)$ with the highest possible angle with the inward normal at $x$.
Henceforth we consider this case.
Write $\alpha=\phi_{xv}-\phi_{xu}$, and note that $\alpha\le 1/({100C\sqrt{n}})$ since $|u-v|\le F(u)/(100\sqrt{n})$. 

Referring to Figure~\ref{fig:overlap},
we next argue that $\nu\ge (c_2-c_1)/(C\sqrt{n})$.
From the sine rule we get $\sin(\delta+\nu) = C\sin(\nu)$, so that 
$\cot(\nu) = (C-\cos(\delta))/\sin(\delta)\le C/\sin(\delta) $. 
Since $x\in A$, we have $\delta\ge (c_2-c_1)/\sqrt{n}$
and therefore $\tan(\nu) \ge (c_2-c_1)/(C\sqrt{n})$
and thus $\nu\ge (c_2-c_1)/(C\sqrt{n})$.

We conclude that
\[
\frac{\cos(\phi_{xv})}{\cos(\phi_{xu})} \ge \frac{\sin(\nu-\alpha)}{\sin(\nu)} \ge 
\frac{\sin((c_2-c_1-1/100)/(C\sqrt{n}))}{\sin((c_2-c_1)/(C\sqrt{n}))} \ge 1-\frac 1{100(c_2-c_1)}>0,
\]
where we use that $\sin(\nu-\alpha)/\sin(\nu)$ is 
increasing in $\nu$ and decreasing in $\alpha$.
Claim~\ref{claim:phi-xv} follows.

This concludes the proof of Lemma~\ref{lem:overlapSB}.
\end{proof}

\section{Conductance}

It is the aim of this section to prove our conductance bound in Proposition~\ref{prop:condSB}.
Apart from the single-step analysis of the previous section, 
a key ingredient is a certain isoperimetric inequality for the boundary of a convex body.
Such inequalities have been studied for several decades, see for instance \cite{MR0397619}.
We need an `integrated' form of this inequality, and we include a proof showing how this lemma follows 
from a classical isoperimetric inequality.
For the state-of-the-art in this area, we refer to the recent work of E.~Milman~\cite{MR2858533}.

\begin{lemma}
\label{lem:milman}
Let $K$ be a convex body in $\Rbb^n$. Suppose $\partial K$ is partitioned into
measurable sets $S_1,S_2,S_3$. We then have, for some constant $c>0$,
\[
\vol(S_3)\ge \frac cD \, d(S_1,S_2) \min(\vol(S_1),\vol(S_2)),
\]
where $d$ denotes the geodesic distance on $\partial K$.
\end{lemma}
\begin{proof}
Recall the definition of the $\epsilon$-extension $A^\epsilon$ of a set $A$ with respect
to the geodesic metric.
Abusing notation, we write $A^\epsilon$ for $A^\epsilon\cap S$.
$\mu^+$ denotes Minkowski's exterior boundary measure, defined through
$
\mu^+(A) = \liminf_{\epsilon\downarrow 0} (\vol(A^\epsilon)-\vol(A))/\epsilon.
$
The isoperimetric constant for manifolds with nonnegative Ricci curvature can be bounded by $c/D$ for some constant $c$ (e.g., \cite{MR2858533}). This yields that, for any $A\subseteq S$,
\[
\mu^+(A) \ge \frac cD \min(\vol(A), \vol(S)-\vol(A)).
\]
For $\epsilon<d(S_1,S_2)$, the inequalities $\vol(S_1^\epsilon)\ge \vol(S_1)$ and $\vol(S)-\vol(S_1^\epsilon)\ge \vol(S_2)$ imply that
\begin{equation}
\label{eq:boundminima}
\min(\vol(S_1^\epsilon), \vol(S)-\vol(S_1^\epsilon)) \ge \min(\vol(S_1),\vol(S_2)).
\end{equation}
The function $x\mapsto \vol(S_1^x)$ is nondecreasing and continuous on $(0,\infty)$.
To see why it is continuous,
let $x>0$ and suppose without loss of generality that $S^x_1$ contains
an $r$-neighborhood $B$ of the origin for some $r>0$. Then
$S_1^{x+\epsilon}=S_1^x+(\epsilon/r) B\subseteq (1+\epsilon/r)S_1^x$, so that
$\vol(S_1^{x})\le \vol(S_1^{x+\epsilon})\le (1+\epsilon/r)^n \vol(S_1^x)$.
Consequently, we have for $x>0$,
\begin{eqnarray*}
\vol(S_1^x) - \vol(S_1) &\ge&
\liminf_{\epsilon\downarrow0}
\left[\frac 1\epsilon\int_x^{x+\epsilon}  \vol(S_1^{\eta}) d\eta
- \frac 1\epsilon\int_0^\epsilon \vol(S_1^{\eta}) d\eta\right]\\&=&
 \liminf_{\epsilon\downarrow0}
\int_0^x \frac 1\epsilon [\vol(S_1)^{\eta+\epsilon} - \vol(S_1)^{\eta}] d\eta
\ge \int_0^x \mu^+(S_1^\eta)d\eta,
\end{eqnarray*}
where the last inequality follows from Fatou's lemma.
Combining the above, we deduce from (\ref{eq:boundminima}) that
\begin{eqnarray*}
\vol(S_3)&\ge& \vol(S_1^{d(S_1,S_2)})-\vol(S_1)\ge
\int_0^{d(S_1,S_2)} \mu^+(S_1^\epsilon) d\epsilon \\&\ge&
\frac cD\int_0^{d(S_1,S_2)} \min(\vol(S_1^\epsilon), \vol(S)-\vol(S_1^\epsilon)) d\epsilon\ge
\frac cD d(S_1,S_2) \min(\vol(S_1),\vol(S_2)),
\end{eqnarray*}
as required.
\end{proof}

We are now ready to prove our conductance bound of Proposition~\ref{prop:condSB}, 
which concludes the proof of our main result.

\begin{proof*}{Proposition~\ref{prop:condSB}}
This part of the proof of Theorem~\ref{thm:main} is quite standard, but we include details here for completeness.

Let $K=S_1\cup S_2$ be a partition into measurable sets.
We will prove that
\begin{eqnarray}
\int_{S_1} P_x(S_2)\,dx &\ge&
\frac{c}{{\cal C} nD} \min\{\vol(S_1), \vol(S_2)\}. \label{CONDUC}
\end{eqnarray}
In this proof, the constant $c$ can vary from line to line. The constant $\kappa$ 
stands for the constant from Lemma~\ref{lem:overlapSB}.
Consider the points that are deep inside these sets, i.e., unlikely to
jump out of the set: 
\[
S_1' = \left\{x \in S_1: P_x(S_2) < \frac{1}{2\kappa}\right\},\quad\quad
S_2' = \left\{x \in S_2: P_x(S_1) < \frac{1}{2\kappa}\right\}.
\]
Set $S_3' = K \setminus S_1'\setminus S_2'$.

Suppose $\vol(S_1') < \vol(S_1)/2$. Then
\[
\int_{S_1} P_x(S_2)\,dx \ge \frac{1}{2\kappa} \vol(S_1\setminus
S_1') \geq \frac{1}{4\kappa}\vol(S_1)
\]
which proves (\ref{CONDUC}).

So we can assume that $\vol(S_1')\geq \vol(S_1)/2$ and similarly
$\vol(S_2') \geq \vol(S_2)/2$. For any $u \in S_1'$ and $v \in S_2'$,
\[
\|P_u-P_v\|_\TV \geq 1 - P_u(S_2) - P_v(S_1) > 1 - \frac{1}{\kappa}.
\]
Thus, by Lemma~\ref{lem:overlapSB}, we must then have
\[
d(u,v) \geq \frac{1}{100\sqrt{n}} \max\{F(u), F(v)\}.
\]
In particular, we have $d(S_1',S_2')\ge \inf_{x\in\partial K} F(x)/(100\sqrt{n})$.

We next apply Lemma~\ref{lem:milman} to obtain
\begin{eqnarray*}
\frac{\vol(S_3')}{\min\{\vol(S_1'), \vol(S_2')\}} &\geq&
\frac{c}{D\sqrt{n}}\inf_{x\in\partial K} F(x) \\
&\ge& \frac{c}{2D\sqrt{n}}\inf_{x\in\partial K} s_\gamma(x),
\end{eqnarray*}
where the last inequality follows from Lemma~\ref{lem:F-s}.
By Lemma~\ref{lem:curvature}, this is bounded from below by $c/({\cal C}nD)$.
Therefore,
\begin{eqnarray*}
\int_{S_1} P_x(S_2)\,dx &\ge& \frac{1}{2}\cdot
\frac{1}{2\kappa}\vol(S_3')\\ 
&\ge& \frac{c}{{\cal C} nD}\min\{\vol(S_1'), \vol(S_2')\}\\
&\ge& \frac{c}{2{\cal C}nD}\min\{\vol(S_1),
\vol(S_2)\}
\end{eqnarray*}
which again proves (\ref{CONDUC}).
\end{proof*}

\section*{Acknowledgments}
ABD gratefully acknowledges the support from NSF grant CMMI-1252878. SV was partially supported by 
NSF award CCF-1217793.
We also thank the referees for their thoughtful comments, and Chang-han Rhee for helpful discussions.

\bibliographystyle{siam}
{\footnotesize
\bibliography{../../bibdb,acg}
}

\end{document}